\documentclass[11pt,twoside]{amsart}
\usepackage{amsmath, amsthm, amscd, amsfonts, amssymb, graphicx, color}
\usepackage[bookmarksnumbered, plainpages]{hyperref}
\newtheorem{thm}{Theorem}[section]

\newtheorem{defn}[thm]{Definition}
\newtheorem{rem}[thm]{\bf{Remark}}

\numberwithin{equation}{section}

\begin{document}
\begin{center}\small{In the name of Allah, the Beneficent, the Merciful.}\end{center}
\vspace{1.5cm}

\begin{center}\textbf{COMPLETE CLASSIFICATIONS OF TWO-DIMENSIONAL GENERAL, COMMUTATIVE, COMMUTATIVE JORDAN, DIVISION AND EVOLUTION REAL ALGEBRAS}\end{center}
\vspace{1cm}
\begin{center} U.Bekbaev \end{center}

\begin{center}Department of Science in Engineering, Faculty of Engineering, IIUM, Kuala Lumpur, Malaysia\end{center}
\begin{center} bekbaev@iium.edu.my\end{center}
\begin{center} MSC(2010): Primary: 15A72; Secondary: 22F50, 20H20, 17A60\end{center}
\begin{center} Keywords: structure constants, division, evolution, Jordan algebras.\end{center}


\begin{abstract}  To describe groups of automorphisms of 2-dimensional real algebras one needs a classification of such algebras up to isomorphism. In this paper a complete classifications of two-dimensional general, commutative, commutative Jordan, division and evolution real algebras are given. In the case of evolution algebras their groups of automorphisms and derivation algebras are described as well.
	
\end{abstract}

\vskip 0.2 true cm

\section{Introduction}
The classification problem of finite dimensional algebras is important in algebra. In this paper we consider such problem for two-dimensional algebras over the field of real numbers $\mathbb{R}$. We provide for the classes of two-dimensional general, commutative, commutative Jordan, division and evolution real algebras the corresponding lists of algebras, given by their matrices of structure constants, such that any nontrivial $2$-dimensional real algebra from any considered class is isomorphic to only one algebra from the corresponding list of algebras. Similar results are stated in \cite{A1,G}.  In \cite{G} case the authors state the existence only whereas the uniqueness can not be guaranteed. In \cite{A1} authors consider the problem over algebraically closed fields. Our approach is similar to of \cite{A1} but different than of \cite{G}.
For further information related to such problems one can see \cite{A2,A3,B1,B2,C,D,H,P1,P3}, in \cite{P2} a complete (basis free) classification of 2-dimensional algebras over any field is presented. 

The next section deals with classification of real two-dimensional general, commutative, commutative Jordan and division algebras. In section 3 we deal with such problem for real $2$-dimensional evolution algebras and describe their groups of automorphisms and derivation algebras.
\section{Classification of two-dimensional general, commutative, commutative Jordan algebras}
In this section we essentially use some calculations provided in \cite{A1} for arbitrary basic field case. For the sake of completeness those calculations are presented here one more time. To classify the main part of two-dimensional algebras we use a particular case of the following result from \cite{B}.  Let $n$, $m$  be any natural numbers, $\tau: (G,V)\rightarrow V$ be a fixed linear algebraic representation of an algebraic subgroup $G$ of $GL(m,\mathbb{F})$ on the $n$-dimensional vector space $V$ over a field $\mathbb{F}$.  Assume that there exists a nonempty $G$-invariant subset $V_0$ of $V$ and an algebraic map $P: V_0\rightarrow G$ such that
\[ P(\tau( g,\mathbf{v}))=P(\mathbf{v}) g^{-1}\] whenever  $\mathbf{v}\in V_0$ and $g\in G$. The following result holds true \cite{B}.

\begin{thm} Elements $\mathbf{u},\mathbf{v}\in V_0$ are $G$-equivalent, that is $\mathbf{u}=\tau( g,\mathbf{v})$ for some $g\in G$, if and only if $\tau(
	P(\mathbf{u}),\mathbf{u})=\tau( P(\mathbf{v}),\mathbf{v})$.\end{thm}

Let $\mathbb{A}$ be any 2-dimensional algebra over $\mathbb{F}$ with multiplication $\cdot$ given by a bilinear map $(\mathbf{u},\mathbf{v})\mapsto \mathbf{u}\cdot \mathbf{v}$ whenever $\mathbf{u}, \mathbf{v}\in \mathbb{A}$. If $e=(e_1,e_2)$ is a
basis for $\mathbb{A}$ as a vector space over $\mathbb{F}$ then one can represent this bilinear map by a matrix \[A=\left(
\begin{array}{cccc}
A^{1}_{1,1} & A^{1}_{1,2}& A^{1}_{2,1}& A^{1}_{2,2} \\
A^{2}_{1,1} & A^{2}_{1,2}& A^{2}_{2,1}& A^{2}_{2,2} \\
\end{array}
\right) \in Mat(2\times 4;\mathbb{F})\] such that \[\mathbf{u}\cdot \mathbf{v}=eA(u\otimes v)\] for any $\mathbf{u}=eu,\mathbf{v}=ev,$
where $u=(u_1, u_2),$ and  $v=(v_1, v_2)$ are column coordinate vectors of $\mathbf{u}$ and $\mathbf{v},$ respectively, $(u\otimes v)=(u_1v_1,u_1v_2,u_2v_1,u_2v_2)$, $e^i\cdot e_j=A^{1}_{i,j}e_1+A^{2}_{i,j}e_2$ whenever $i,j=1,2$.
So the algebra $\mathbb{A}$ is presented by the matrix $A\in Mat(2\times 4;\mathbb{F})$ (called the matrix of MSC of $\mathbb{A}$ with respect to the basis $e$).

If $e'=(e'_1,e'_2)$ is also a basis for $\mathbb{A}$, $g\in G=GL(2,\mathbb{F})$, $e'g=e$ and  $\mathbf{u}\cdot \mathbf{v}=e'B(u'\otimes v')$, where $\mathbf{u}=e'u',\mathbf{v}=e'v'$, then
\[\mathbf{u}\cdot \mathbf{v}=eA(u\otimes v)=e'B(u'\otimes v')=eg^{-1}B(gu\otimes gv)=eg^{-1}B(g\otimes g)(u\otimes v)\] as far as $\mathbf{u}=eu=e'u'=eg^{-1}u',\mathbf{v}=ev=e'v'=eg^{-1}v'$.
Therefore the equality
\[B=gA(g^{-1})^{\otimes 2}\] is valid, where for $g^{-1}=\left(\begin{array}{cccc} \xi_1& \eta_1\\ \xi_2& \eta_2\end{array}\right)$ one has
\[(g^{-1})^{\otimes 2}=g^{-1}\otimes g^{-1}=\left(\begin{array}{cccc} \xi _1^2 & \xi _1 \eta _1 & \xi _1 \eta _1 & \eta _1^2 \\
\xi _1 \xi _2 & \xi _1 \eta _2 & \xi _2 \eta _1 & \eta _1 \eta _2 \\
\xi _1 \xi _2 & \xi _2 \eta _1 & \xi _1 \eta _2 & \eta _1 \eta _2 \\
\xi _2^2 & \xi _2 \eta _2 & \xi _2 \eta _2 & \eta _2^2
\end{array}
\right).\]

\begin{defn} Two-dimensional algebras $\mathbb{A}$, $\mathbb{B}$, given by
their matrices of structural constants $A$, $B$, are said to be isomorphic if $B=gA(g^{-1})^{\otimes 2}$ holds true for some $g\in GL(2,\mathbb{F})$.\end{defn}

Note that the following identities \begin{equation} \label{2.1} Tr_1(gA(g^{-1})^{\otimes 2})=Tr_1(A)g^{-1},\ \ Tr_2(gA(g^{-1})^{\otimes 2})=Tr_2(A)g^{-1}\end{equation} hold true whenever $A\in Mat(2\times 4,\mathbb{F}), g\in GL(2,\mathbb{F})$, where \[Tr_1(A)=(A^1_{1,1}+A^2_{2,1}, A^1_{1,2}+A^2_{2,2}),\ Tr_2(A)=(A^1_{1,1}+A^2_{1,2}, A^1_{2,1}+A^2_{2,2})\] are row vectors.

We divide $Mat(2\times 4,\mathbb{R})$ into the following five disjoint subsets:
\begin{itemize}
	\item[1.] All $A$ for which the system $\{Tr_1(A),Tr_2(A)\}$ is linear independent.
	\item[2.] All $A$ for which the system $\{Tr_1(A),Tr_2(A)\}$ is linear dependent and $Tr_1(A),Tr_2(A)$ are nonzero vectors.
	\item[3.] All $A$ for which $Tr_1(A)$ is nonzero vector and $Tr_2(A)=(0,0)$.
	\item[4.] All $A$ for which $Tr_1(A)=(0,0)$ and $Tr_2(A)$ is nonzero vector.
	\item[5.] All $A$ for which$Tr_1(A)=Tr_2(A)=(0,0)$.
\end{itemize}

Due to (\ref{2.1}) it is clear that algebras with matrices from these different classes can not be isomorphic. We deal with each of these subsets separately.
Further, for the simplicity, we use the notation \[A=\left(\begin{array}{cccc} \alpha_1 & \alpha_2 & \alpha_3 &\alpha_4\\ \beta_1 & \beta_2 & \beta_3 &\beta_4\end{array}\right),\] where
$\alpha_1, \alpha_2, \alpha_3, \alpha_4, \beta_1, \beta_2, \beta_3, \beta_4$ stand for any elements of $\mathbb{R}$.

Now we are going to prove the following classification theorem.

{\begin{thm} Any non-trivial 2-dimensional real algebra is isomorphic to only one of the following listed, by their matrices of structure constants, algebras:
	\[A_{1,r}(\mathbf{c})=\left(
	\begin{array}{cccc}
	\alpha_1 & \alpha_2 &\alpha_2+1 & \alpha_4 \\
	\beta_1 & -\alpha_1 & -\alpha_1+1 & -\alpha_2
	\end{array}\right),\ \mbox{where}\ \mathbf{c}=(\alpha_1, \alpha_2, \alpha_4, \beta_1)\in \mathbb{R}^4,\]
	\[A_{2,r}(\mathbf{c})=\left(
	\begin{array}{cccc}
	\alpha_1 & 0 & 0 & 1 \\
	\beta _1& \beta _2& 1-\alpha_1&0
	\end{array}\right), \ \mbox{where}\ \beta_1\geq 0,\ \mathbf{c}=(\alpha_1, \beta_1, \beta_2)\in \mathbb{R}^3,\]
	\[A_{3,r}(\mathbf{c})=\left(
	\begin{array}{cccc}
	\alpha_1 & 0 & 0 & -1 \\
	\beta _1& \beta _2& 1-\alpha_1&0
	\end{array}\right), \ \mbox{where}\ \beta_1\geq 0,\ \mathbf{c}=(\alpha_1, \beta_1, \beta_2)\in \mathbb{R}^3,\]
	\[A_{4,r}(\mathbf{c})=\left(
	\begin{array}{cccc}
	0 & 1 & 1 & 0 \\
	\beta _1& \beta _2 & 1&-1
	\end{array}\right),\ \mbox{where}\ \mathbf{c}=(\beta_1, \beta_2)\in \mathbb{R}^2,\]
	\[A_{5,r}(\mathbf{c})=\left(
	\begin{array}{cccc}
	\alpha _1 & 0 & 0 & 0 \\
	0 & \beta _2& 1-\alpha _1&0
	\end{array}\right),\ \mbox{where}\ \mathbf{c}=(\alpha_1, \beta_2)\in \mathbb{R}^2,\]
	\[A_{6,r}(\mathbf{c})=\left(
	\begin{array}{cccc}
	\alpha_1& 0 & 0 & 0 \\
	1 & 2\alpha_1-1 & 1-\alpha_1&0
	\end{array}\right),\ \mbox{where}\ \mathbf{c}=\alpha_1\in \mathbb{R},\]
	\[A_{7,r}(\mathbf{c})=\left(
	\begin{array}{cccc}
	\alpha_1 & 0 & 0 & 1 \\
	\beta _1& 1-\alpha_1 & -\alpha_1&0
	\end{array}\right),\ \mbox{where}\ \beta_1\geq 0,\ \mathbf{c}=(\alpha_1, \beta_1)\in \mathbb{R}^2,\]
	\[A_{8,r}(\mathbf{c})=\left(
	\begin{array}{cccc}
	\alpha_1 & 0 & 0 & -1 \\
	\beta _1& 1-\alpha_1 & -\alpha_1&0
	\end{array}\right),\ \mbox{where}\ \beta_1\geq 0,\ \mathbf{c}=(\alpha_1, \beta_1)\in \mathbb{R}^2,\]
	\[A_{9,r}(\mathbf{c})=\left(
	\begin{array}{cccc}
	0 & 1 & 1 & 0 \\
	\beta_1& 1& 0&-1
	\end{array}\right),\ \mbox{where}\ \mathbf{c}=\beta_1\in \mathbb{R},\]
	\[A_{10,r}(\mathbf{c})=\left(
	\begin{array}{cccc}
	\alpha_1 & 0 & 0 & 0 \\
	0 & 1-\alpha_1 & -\alpha_1&0
	\end{array}\right),\ \mbox{where}\ \mathbf{c}=\alpha_1\in\mathbb{R},\]
	\[ A_{11,r}=\left(
	\begin{array}{cccc}
	\frac{1}{3}& 0 & 0 & 0 \\
	1 & \frac{2}{3} & -\frac{1}{3}&0
	\end{array}\right),\
	\ A_{12,r}=\left(
	\begin{array}{cccc}
	0 & 1 & 1 & 0 \\
	1 &0&0 &-1
	\end{array}
	\right),\]
	\[ A_{13,r}=\left(
	\begin{array}{cccc}
	0 & 1 & 1 & 0 \\
	-1 &0&0 &-1\end{array}
	\right),\
	\ A_{14,r}=\left(
	\begin{array}{cccc}
	0 & 1 & 1 & 0 \\
	0 &0&0 &-1\end{array}
	\right),\
	\ A_{15,r}=\left(
	\begin{array}{cccc}
	0 & 0 & 0 & 0 \\
	1 &0& 0 &0
	\end{array}
	\right).\]\end{thm}

\begin{proof}
{\bf The first subset case.} In this case let $P(A)$ stand for a nonsingular matrix $\left(\begin{array}{cc} \alpha_1+\beta_3& \alpha_2+\beta_4\\ \alpha_1+\beta_2 & \alpha_3+\beta_4\end{array}\right)$  with rows
\[Tr_1(A)=(\alpha _1+\beta _3, \alpha _2+\beta _4),\ Tr_2(A)=(\alpha _1+\beta _2, \alpha _3+\beta _4).\] Due to (\ref{2.1}) we have identity \[P(g A(g^{-1})^{\otimes 2})= P(A)g^{-1},\] which means that one can apply Theorem 2.1 in our $V=\mathbb{R}^8$, $\tau(g,A)= g A(g^{-1})^{\otimes 2}$ case. In this case for $V_0$ one can take \[ V_0=\{A\in Mat(2\times 4,\mathbb{R}): \det(P(A))=\alpha_1(\alpha_3-\alpha_2)+\beta_4(\beta_3-\beta_2)+\alpha_3\beta_3-\alpha_2\beta_2\neq 0.\}\] Therefore,
two-dimensional algebras $\mathbb{A}$, $\mathbb{B}$, given by
their matrices of structure constants $A$, $B \in V_0$, are isomorphic if and only if the equality \[P(B)B(P(B)^{-1}\otimes P(B)^{-1})=P(A)A(P(A)^{-1}\otimes P(A)^{-1})\] holds true.

For $\xi_1=\frac{\alpha_3+\beta_4}{\Delta}$, $\xi_2=-\frac{\alpha_1+\beta_2}{\Delta}$, $\eta_1=-\frac{\alpha_2+\beta_4}{\Delta}$, $\eta_2=\frac{\alpha_1+\beta_3}{\Delta}$, where $\Delta$ stands for $\det(P(A))$, one has
\[P(A)^{-1}=\left(\begin{array}{cc} \xi_1& \eta_1\\ \xi_2& \eta_2\end{array}\right)\] and
$A'=\left(\begin{array}{cccc} \alpha'_1 & \alpha'_2 & \alpha'_3 &\alpha'_4\\ \beta'_1 & \beta'_2 & \beta'_3 &\beta'_4\end{array}\right)=P(A)A(P(A)^{-1}\otimes P(A)^{-1})$ is matrix consisting of columns
\[\left(\begin{array}{c}
-\frac{\beta _1 \eta _1 \xi _1^2}{\Delta }+\frac{\alpha _1 \eta _2 \xi _1^2}{\Delta }+\frac{2 \alpha _1 \eta _1 \xi _1 \xi _2}{\Delta }-\frac{2 \beta _4 \eta _2 \xi _1 \xi _2}{\Delta }-\frac{2 \eta _1 \eta _2 \xi _1 \xi _2}{\Delta ^2}+\frac{\eta _2 \xi _1^2 \xi _2}{\Delta ^2}-\frac{\beta _4 \eta _1 \xi _2^2}{\Delta }+\frac{\alpha _4 \eta _2 \xi _2^2}{\Delta }+\frac{\eta _1 \xi _1 \xi _2^2}{\Delta ^2}  \\
\frac{\beta _1 \xi _1^3}{\Delta }-\frac{3 \alpha _1 \xi _1^2 \xi _2}{\Delta }+\frac{\eta _2 \xi _1^2 \xi _2}{\Delta ^2}+\frac{3 \beta _4 \xi _1 \xi _2^2}{\Delta }+\frac{\eta _1 \xi _1 \xi _2^2}{\Delta ^2}-\frac{2 \xi _1^2 \xi _2^2}{\Delta ^2}-\frac{\alpha _4 \xi _2^3}{\Delta }
\end{array}
\right),\]
\[\left(\begin{array}{c}-\frac{\beta _1 \eta _1^2 \xi _1}{\Delta }+\frac{2 \alpha _1 \eta _1 \eta _2 \xi _1}{\Delta }-\frac{\beta _4 \eta _2^2 \xi _1}{\Delta }-\frac{\eta _1 \eta _2^2 \xi _1}{\Delta ^2}+\frac{\alpha _1 \eta _1^2 \xi _2}{\Delta }-\frac{2 \beta _4 \eta _1 \eta _2 \xi _2}{\Delta }-\frac{\eta _1^2 \eta _2 \xi _2}{\Delta ^2}+\frac{\alpha _4 \eta _2^2 \xi _2}{\Delta }+\frac{2 \eta _1 \eta _2 \xi _1 \xi _2}{\Delta ^2}\\
\frac{\beta _1 \eta _1 \xi _1^2}{\Delta }-\frac{\alpha _1 \eta _2 \xi _1^2}{\Delta }-\frac{2 \alpha _1 \eta _1 \xi _1 \xi _2}{\Delta }+\frac{2 \beta _4 \eta _2 \xi _1 \xi _2}{\Delta }+\frac{2 \eta _1 \eta _2 \xi _1 \xi _2}{\Delta ^2}-\frac{\eta _2 \xi _1^2 \xi _2}{\Delta ^2}+\frac{\beta _4 \eta _1 \xi _2^2}{\Delta }-\frac{\alpha _4 \eta _2 \xi _2^2}{\Delta }-\frac{\eta _1 \xi _1 \xi _2^2}{\Delta ^2}
\end{array}
\right),\]

\[\left(\begin{array}{c} -\frac{\beta _1 \eta _1^2 \xi _1}{\Delta }+\frac{2 \alpha _1 \eta _1 \eta _2 \xi _1}{\Delta }-\frac{\beta _4 \eta _2^2 \xi _1}{\Delta }-\frac{\eta _1 \eta _2^2 \xi _1}{\Delta ^2}+\frac{\eta _2^2 \xi _1^2}{\Delta ^2}+\frac{\alpha _1 \eta _1^2 \xi _2}{\Delta }-\frac{2 \beta _4 \eta _1 \eta _2 \xi _2}{\Delta }-\frac{\eta _1^2 \eta _2 \xi _2}{\Delta ^2}+\frac{\alpha _4 \eta _2^2 \xi _2}{\Delta }+\frac{\eta _1^2 \xi _2^2}{\Delta ^2}\\
\frac{\beta _1 \eta _1 \xi _1^2}{\Delta }-\frac{\alpha _1 \eta _2 \xi _1^2}{\Delta }+\frac{\eta _2^2 \xi _1^2}{\Delta ^2}-\frac{2 \alpha _1 \eta _1 \xi _1 \xi _2}{\Delta }+\frac{2 \beta _4 \eta _2 \xi _1 \xi _2}{\Delta }-\frac{\eta _2 \xi _1^2 \xi _2}{\Delta ^2}+\frac{\beta _4 \eta _1 \xi _2^2}{\Delta }+\frac{\eta _1^2 \xi _2^2}{\Delta ^2}-\frac{\alpha _4 \eta _2 \xi _2^2}{\Delta }-\frac{\eta _1 \xi _1 \xi _2^2}{\Delta ^2}\end{array}
\right),\]

\[\left(\begin{array}{c} -\frac{\beta _1 \eta _1^3}{\Delta }+\frac{3 \alpha _1 \eta _1^2 \eta _2}{\Delta }-\frac{3 \beta _4 \eta _1 \eta _2^2}{\Delta }-\frac{2 \eta _1^2 \eta _2^2}{\Delta ^2}+\frac{\alpha _4 \eta _2^3}{\Delta }+\frac{\eta _1 \eta _2^2 \xi _1}{\Delta ^2}+\frac{\eta _1^2 \eta _2 \xi _2}{\Delta ^2}\\
\frac{\beta _1 \eta _1^2 \xi _1}{\Delta }-\frac{2 \alpha _1 \eta _1 \eta _2 \xi _1}{\Delta }+\frac{\beta _4 \eta _2^2 \xi _1}{\Delta }+\frac{\eta _1 \eta _2^2 \xi _1}{\Delta ^2}-\frac{\alpha _1 \eta _1^2 \xi _2}{\Delta }+\frac{2 \beta _4 \eta _1 \eta _2 \xi _2}{\Delta }+\frac{\eta _1^2 \eta _2 \xi _2}{\Delta ^2}-\frac{\alpha _4 \eta _2^2 \xi _2}{\Delta }-\frac{2 \eta _1 \eta _2 \xi _1 \xi _2}{\Delta ^2}\end{array}
\right).\]

From the above it can be easily seen that $\alpha'_1=-\beta'_2$, $\alpha'_3=\alpha'_2+1$, $\beta'_3=\beta'_2+1$ and $\beta'_4=-\alpha'_2,$ i.e.
\[A'=\left(
\begin{array}{cccc}
\alpha'_1 & \alpha'_2 &\alpha'_2+1 & \alpha'_4 \\
\beta'_1 & -\alpha'_1 & -\alpha'_1+1 & -\alpha'_2
\end{array}\right).\]
Therefore, the main role in the finding of $A'$ is played by the functions

$\alpha'_1=-\frac{\beta _1 \eta _1 \xi _1^2}{\Delta }+\frac{\alpha _1 \eta _2 \xi _1^2}{\Delta }+\frac{2 \alpha _1 \eta _1 \xi _1 \xi _2}{\Delta }-\frac{2 \beta _4 \eta _2 \xi _1 \xi _2}{\Delta }-\frac{2 \eta _1 \eta _2 \xi _1 \xi _2}{\Delta ^2}+\frac{\eta _2 \xi _1^2 \xi _2}{\Delta ^2}-\frac{\beta _4 \eta _1 \xi _2^2}{\Delta }+\frac{\alpha _4 \eta _2 \xi _2^2}{\Delta }+\frac{\eta _1 \xi _1 \xi _2^2}{\Delta ^2},$

$\beta'_1=\frac{\beta _1 \xi _1^3}{\Delta }-\frac{3 \alpha _1 \xi _1^2 \xi _2}{\Delta }+\frac{\eta _2 \xi _1^2 \xi _2}{\Delta ^2}+\frac{3 \beta _4 \xi _1 \xi _2^2}{\Delta }+\frac{\eta _1 \xi _1 \xi _2^2}{\Delta ^2}-\frac{2 \xi _1^2 \xi _2^2}{\Delta ^2}-\frac{\alpha _4 \xi _2^3}{\Delta }
,$

$\alpha'_2=-\frac{\beta _1 \eta _1^2 \xi _1}{\Delta }+\frac{2 \alpha _1 \eta _1 \eta _2 \xi _1}{\Delta }-\frac{\beta _4 \eta _2^2 \xi _1}{\Delta }-\frac{\eta _1 \eta _2^2 \xi _1}{\Delta ^2}+\frac{\alpha _1 \eta _1^2 \xi _2}{\Delta }-\frac{2 \beta _4 \eta _1 \eta _2 \xi _2}{\Delta }-\frac{\eta _1^2 \eta _2 \xi _2}{\Delta ^2}+\frac{\alpha _4 \eta _2^2 \xi _2}{\Delta }+\frac{2 \eta _1 \eta _2 \xi _1 \xi _2}{\Delta ^2},$

$\alpha'_4=-\frac{\beta _1 \eta _1^3}{\Delta }+\frac{3 \alpha _1 \eta _1^2 \eta _2}{\Delta }-\frac{3 \beta _4 \eta _1 \eta _2^2}{\Delta }-\frac{2 \eta _1^2 \eta _2^2}{\Delta ^2}+\frac{\alpha _4 \eta _2^3}{\Delta }+\frac{\eta _1 \eta _2^2 \xi _1}{\Delta ^2}+\frac{\eta _1^2 \eta _2 \xi _2}{\Delta ^2}.$

It can be verified that these functions can take any values in $\mathbb{R}.$ Therefore, the values of $\alpha'_1,\beta'_1, \alpha'_2, \alpha'_4$ also can be any elements in $\mathbb{R}$. Using obvious re notations one can list the following non-isomorphic ``canonical'' algebras from the first subset, given by their MSCs as \[
A_{1,r}(\mathbf{c})=\left(
\begin{array}{cccc}
\alpha_1 & \alpha_2 &\alpha_2+1 & \alpha_4 \\
\beta_1 & -\alpha_1 & -\alpha_1+1 & -\alpha_2
\end{array}\right),\ \mbox{where}\ \mathbf{c}=(\alpha_1, \alpha_2, \alpha_4, \beta_1) \in \mathbb{R}^4.
\]
For any algebra from the first subset there exists a unique algebra from   $\{A_{1}(\mathbf{c}):\ \mathbf{c}\in \mathbb{R}^4\}$ which is isomorphic to that algebra.

{\bf The second and third subset cases.} In these cases one can make $Tr_1(A)g=(1,0)$. It implies that $Tr_2(A)g=(\lambda,0)$. Therefore, we can assume that  $Tr_1(A)=(\alpha _1+\beta _3,\alpha _2+\beta _4)=(1,0), Tr_2(A)=( \lambda(\alpha _1+\beta _3),\lambda(\alpha _2+\beta _4)=(\lambda,0)$. Here if $\lambda$ is zero we can also cover the third subset. So let us consider
\[A= \left(
\begin{array}{cccc}
\alpha _1 & \alpha _2 & \lambda  \alpha _2+(\lambda -1)\beta _4 & \alpha _4 \\
\beta _1 & (\lambda -1)\alpha _1+\lambda  \beta _3 & \beta _3 & \beta _4
\end{array}
\right)=\left(
\begin{array}{cccc}
\alpha _1 & \alpha _2 & \alpha _2 & \alpha _4 \\
\beta _1 & \lambda-\alpha _1 & 1-\alpha _1& -\alpha _2
\end{array}\right),\] with respect to $g\in GL(2,\mathbb{R})$ of the form \[g^{-1}=\left(\begin{array}{cc} 1& 0\\ \xi_2& \eta_2\end{array}\right),\] as far as
$(1,0)g^{-1}=(1,0)$ if and only if $g^{-1}$ is of the above form. In this case for the entries of $A'$ we have $A'=\left(
\begin{array}{cccc}
\alpha'_1 & \alpha'_2 & \alpha'_2 & \alpha'_4 \\
\beta'_1 & \lambda-\alpha'_1 & 1-\alpha'_1& -\alpha'_2
\end{array}\right)=gA(g^{-1})^{\otimes 2}$, where $\lambda$ is same for the $A$ and $A'$,
one has\\
$\alpha '_1=\frac{1}{\Delta }(\alpha _1 \eta _2 +2\alpha _2\eta _2\xi _2+\alpha _4 \eta _2 \xi _2^2)=\alpha _1+2\alpha _2\xi _2+\alpha _4\xi _2^2,$\\
$\alpha '_2=\frac{-1}{\Delta }(-\alpha _2\eta^2_2 -\alpha _4 \eta _2^2 \xi _2)=(\alpha _2+\alpha _4\xi _2)\eta _2,$\\
$\alpha '_4=\frac{-1}{\Delta }(\beta _1 \eta _1^3+[(\lambda-2)\alpha _1+(1+\lambda)\beta _3]\eta _1^2 \eta _2-[(1+\lambda)\alpha _2+(\lambda-2)\beta _4]\eta _1 \eta _2^2-\alpha _4 \eta _2^3)=\alpha _4 \eta_2^2,$\\
$\beta'_1=\frac{1}{\Delta }(\beta _1 \xi _1^3+[(\lambda-2)\alpha _1 +(\lambda+1)\beta _3]\xi _1^2 \xi _2-[(1+\lambda)\alpha _2+(\lambda-2)\beta _4] \xi _1 \xi _2^2-\alpha _4 \xi _2^3)=
\frac{\beta _1+(1+\lambda-3 \alpha _1)\xi _2-3 \alpha _2 \xi _2^2-\alpha _4 \xi _2^3}{\eta_2}.$

So it is enough to consider the system
\[\alpha '_1=\alpha _1+2 \alpha _2 \xi _2+\alpha _4 \xi _2^2,\ \
\alpha '_2=(\alpha _2+\alpha_4\xi_2) \eta _2,\ \
\alpha '_4=\alpha _4 \eta _2^2,\ \
\beta '_1=\frac{\beta _1+(1+\lambda-3\alpha _1)\xi _2-3 \alpha _2 \xi _2^2-\alpha _4 \xi _2^3}{\eta _2}.\]

1. $\alpha_4\neq 0$ case. In this case one can make \[\alpha '_2=0, \alpha '_4=\pm 1, \alpha '_1=\alpha_1-\frac{\alpha^2_2}{\alpha_4}, \beta '_1=\pm\sqrt{|\alpha_4|}\left(\beta _1-\left(1+\lambda-3\alpha _1\right)\frac{\alpha_2}{\alpha_4}-2\frac{\alpha^3_2}{\alpha^2_4}\right)\]
and
once again using re notation one can represent the corresponding ``canonical'' MSCs as 
\[ A_{2,r}(\mathbf{c})=\left(
\begin{array}{cccc}
\alpha_1 & 0 & 0 & 1 \\
\beta _1& \beta_2 & 1-\alpha_1&0
\end{array}\right),\ A_{3,r}(\mathbf{c})=\left(
\begin{array}{cccc}
\alpha_1 & 0 & 0 & -1 \\
\beta _1& \beta_2 & 1-\alpha_1&0
\end{array}\right),\] where $\mathbf{c}=(\alpha_1, \beta_1, \beta_2)\in \mathbb{R}^3,$ as far as the functions $\alpha_1-\frac{\alpha^2_2}{\alpha_4}$, $\sqrt{|\alpha_4|}(\beta _1-(1+\lambda-3\alpha _1)\frac{\alpha_2}{\alpha_4}-2\frac{\alpha^3_2}{\alpha^2_4})$, $\lambda-(\alpha_1-\frac{\alpha^2_2}{\alpha_4} ) $ may have any values in $\mathbb{R}$. Note that there are isomorphisms
$\left(
\begin{array}{cccc}
\alpha_1 & 0 & 0 & 1 \\
\beta _1& \beta_2 & 1-\alpha_1&0
\end{array}\right)\simeq$ \[\left(
\begin{array}{cccc}
\alpha_1 & 0 & 0 & 1 \\
-\beta _1& \beta_2 & 1-\alpha_1&0
\end{array}\right),\ \left(
\begin{array}{cccc}
\alpha_1 & 0 & 0 & -1 \\
\beta _1& \beta_2 & 1-\alpha_1&0
\end{array}\right)\simeq\left(
\begin{array}{cccc}
\alpha_1 & 0 & 0 & -1 \\
-\beta _1& \beta_2 & 1-\alpha_1&0
\end{array}\right).\]
2. $\alpha _4=0$ case.

2 - a). If $\alpha _2\neq 0$ then one can make
$\alpha '_1=0$, $\alpha '_2=1$, $\beta '_1=\alpha_2\beta _1-\frac{2+2\lambda-3\alpha _1}{4}\alpha_1$ to get the following set of canonical matrices of structure constants
\[ A_{4,r}(\mathbf{c})=\left(
\begin{array}{cccc}
0 & 1 & 1 & 0 \\
\beta _1& \beta_2 & 1&-1
\end{array}\right),\ \mbox{where}\ \mathbf{c}=(\beta_1, \beta_2)\in \mathbb{R}^2.\]

2 - b). If $\alpha _2=0$ then $\alpha '_1=\alpha _1,$ $\alpha '_2=0,$ $\alpha '_4=0,$ $\beta '_1=\frac{\beta _1+(1+\lambda-3\alpha _1)\xi _2}{\eta _2}.$

2 - b) - 1. If $1+\lambda-3\alpha _1\neq 0$, that is $\lambda-\alpha _1\neq 2\alpha_1-1$,  one can make $\beta'_1=0$ to get
\[A_{5,r}(\mathbf{c})=\left(
\begin{array}{cccc}
\alpha _1 & 0 & 0 & 0 \\
0 & \beta_2 & 1-\alpha _1&0
\end{array}\right),\ \mbox{where}\ \mathbf{c}=(\alpha_1, \beta_2) \in \mathbb{R}^2, \ \mbox{with}\  \beta_2\neq 2\alpha_1-1.\]

2 - b) - 2. If $1+\lambda-3\alpha _1=0$ and $\beta_1\neq 0$ one can make $\beta'_1=1$ to get
\[A_{6,r}(c)=\left(
\begin{array}{cccc}
\alpha_1 & 0 & 0 & 0 \\
1 & 2\alpha_1-1 & 1-\alpha_1&0
\end{array}\right),\ \mbox{where}\ c=\alpha_1\in \mathbb{R},\] if $\beta_1=0$ one has $\lambda-\alpha _1= 2\alpha_1-1$ and therefore
$A'=\left(
\begin{array}{cccc}
\alpha _1 & 0 & 0 & 0 \\
0 & 2\alpha _1-1 & 1-\alpha _1&0
\end{array}\right)$\\ which is $A_{5,r}(\mathbf{c})$ with $\beta_2=2\alpha_1-1$.

{\bf The fourth subset case.}
By the similar justification as in the second and third subsets case it is enough to consider  \[ A=\left(
\begin{array}{cccc}
\alpha _1 & \alpha _2 & \alpha _2 & \alpha _4 \\
\beta _1 & 1-\alpha _1 & -\alpha _1 & -\alpha _2
\end{array}
\right),\ \mbox{and}\ \ g^{-1}=\left(\begin{array}{cc} 1& 0\\ \xi_2& \eta_2\end{array}\right),\
\mbox{where one has}\]
\[\alpha'_1= \alpha _1+2 \alpha _2 \xi _2+\alpha _4 \xi _2^2 ,\ \
\alpha'_2= (\alpha _2+\alpha _4\xi _2)\eta_2,\ \
\alpha'_4=\alpha _4 \eta _2^2,\ \
\beta'_1=\frac{\beta _1+\xi _2-3 \alpha _1 \xi _2-3 \alpha _2 \xi _2^2-\alpha _4 \xi _2^3}{\eta _2}.\]

Therefore we get the canonical MSCs as follows
\[A_{7,r}(\mathbf{c})=\left(
\begin{array}{cccc}
\alpha_1 & 0 & 0 & 1 \\
\beta _1& 1-\alpha_1 & -\alpha_1&0
\end{array}\right)\simeq\left(
\begin{array}{cccc}
\alpha_1 & 0 & 0 & 1 \\
-\beta _1& 1-\alpha_1 & -\alpha_1&0
\end{array}\right),\] \[ A_{8,r}(\mathbf{c})=\left(
\begin{array}{cccc}
\alpha_1 & 0 & 0 & -1 \\
\beta _1& 1-\alpha_1 & -\alpha_1&0
\end{array}\right)\simeq\left(
\begin{array}{cccc}
\alpha_1 & 0 & 0 & -1 \\
-\beta _1& 1-\alpha_1 & -\alpha_1&0
\end{array}\right),\ \mbox{where}\ \mathbf{c}=(\alpha_1, \beta_1) \in \mathbb{R}^2,\]
(they are $A_{2,r}(\mathbf{c})$ and $A_{3,r}(\mathbf{c})$, respectively, where the $2^{nd}$ and $3^{rd}$ columns are interchanged and $\alpha_1+\beta_2=0$) or
\[ A_{9,r}(\mathbf{c})=\left(
\begin{array}{cccc}
0 & 1 & 1 & 0 \\
\beta _1& 1& 0&-1
\end{array}\right),\ \mbox{where}\ \mathbf{c}=\beta_1\in \mathbb{R},\]
(it is $A_{4,r}(\mathbf{\mathbf{c}})$, where the $2^{nd}$ and $3^{rd}$ columns are interchanged and $\beta_2=0$) or
\[A_{10,r}(\mathbf{c})=\left(
\begin{array}{cccc}
\alpha_1 & 0 & 0 & 0 \\
0 & 1-\alpha_1 & -\alpha_1&0
\end{array}\right),\ \mbox{where}\ \mathbf{c}=\alpha_1\in \mathbb{F},\]
(it is $A_{5,r}(\mathbf{c})$, where the $2^{nd}$ and $3^{rd}$ columns are interchanged and $\alpha_1+\beta_2=0$) or
\[A_{11,r}=\left(
\begin{array}{cccc}
\frac{1}{3} & 0 & 0 & 0 \\
1 & \frac{2}{3} & -\frac{1}{3} &0
\end{array}\right)\] (it is $A_{6,r}(\mathbf{c})$, where the $2^{nd}$ and $3^{rd}$ columns are interchanged and $3\alpha_1-1=0$).

{\bf The fifth subset case.} In this case
\[A=\left(
\begin{array}{cccc}
\alpha _1 & \alpha _2 & \alpha _2 & \alpha _4 \\
\beta _1 & -\alpha _1 & -\alpha _1 & -\alpha _2
\end{array}
\right),\ g^{-1}=\left(\begin{array}{cc} \xi_1& \eta_1\\ \xi_2& \eta_2\end{array}\right)\mbox{and}\
A'=\left(
\begin{array}{cccc}
\alpha' _1 & \alpha' _2 & \alpha' _2 & \alpha' _4 \\
\beta' _1 & -\alpha' _1 & -\alpha' _1 & -\alpha' _2
\end{array}
\right)\]
such that:

$\alpha' _1=\frac{1}{\Delta }\left(-\beta _1 \eta _1 \xi _1^2+\alpha _1 \eta _2 \xi _1^2+2 \alpha _1 \eta _1 \xi _1 \xi _2+2 \alpha _2 \eta _2 \xi _1 \xi _2+\alpha _2 \eta _1 \xi _2^2+\alpha _4 \eta _2 \xi _2^2\right),$

$\alpha' _2=\frac{-1}{\Delta }\left(\beta _1 \eta _1^2 \xi _1-2 \alpha _1 \eta _1 \eta _2 \xi _1-\alpha _2 \eta _2^2 \xi _1-\alpha _1 \eta _1^2 \xi _2-2 \alpha _2 \eta _1 \eta _2 \xi _2-\alpha _4 \eta _2^2 \xi _2\right),$

$\alpha' _4=\frac{-1}{\Delta }\left(\beta _1 \eta _1^3-3 \alpha _1 \eta _1^2 \eta _2-3 \alpha _2 \eta _1 \eta _2^2-\alpha _4 \eta _2^3\right),$

$\beta' _1=\frac{1}{\Delta }\left(\beta _1 \xi _1^3-3 \alpha _1 \xi _1^2 \xi _2-3 \alpha _2 \xi _1 \xi _2^2-\alpha _4 \xi _2^3\right).$

If $\alpha _4\neq 0$ by making $\frac{\eta_2}{\eta_1}$ equal to any root of the polynomial $\beta _1-3 \alpha _1t-3 \alpha _2t^2-\alpha _4t^3$ one can make $\alpha'_4=0$. Therefore, further it is assumed that $\alpha _4= 0$.

Let us consider $g$ with $\eta_1=0$ to have $\alpha'_4=0$. In this case $\Delta=\det(g)=\xi_1\eta_2$ and
\[\alpha' _1=\xi _1\left(\alpha _1+2\alpha _2\frac{\xi_2}{\xi_1}\right),\ \ \alpha' _2=\eta_2\alpha _2,\ \  \beta' _1=\frac{ \xi _1^2}{\eta_2}\left(\beta _1-3 \alpha _1 \frac{\xi_2}{\xi_1}-3 \alpha _2(\frac{\xi_2}{\xi_1})^2\right).\]

If $\alpha_2\neq 0$ one can consider
$\frac{\xi_2}{\xi_1}=\frac{-\alpha_1}{2\alpha_2}$,  $\eta_2=\alpha^{-1}_2$ to get
$\alpha' _1=0$, $\alpha' _2=1$ and
$\beta'_1=\xi_1^2\frac{3\alpha^2_1+4\alpha^2_2}{4}$. Therefore one can make $\beta' _1$ equal to $\pm 1$ or 0, depending on $3\alpha^2_1+4\alpha^2_2$, to get
\[A_{12,r}=\left(
\begin{array}{cccc}
0 & 1 & 1 & 0 \\
1 &0&0 &-1\end{array}
\right),\ \mbox{or}\ A_{13,r}=\left(
\begin{array}{cccc}
0 & 1 & 1 & 0 \\
-1 &0&0 &-1\end{array}
\right),\ \mbox{ or}\ A_{14,r}=\left(
\begin{array}{cccc}
0 & 1 & 1 & 0 \\
0 &0&0 &-1\end{array}
\right).\] 

If $\alpha_2= 0$ then $\alpha' _2=\alpha' _4=0$ and
$\alpha' _1=\xi _1\alpha _1,$
$\beta' _1=\frac{\xi^2_1}{\eta_2 }(\beta _1-3\alpha_1\frac{\xi_2}{\xi_1}).$ Therefore if
$\alpha_1\neq 0$ one can make $\alpha'_1=1$,$\beta' _1=0$ to get
\[A'=\left(
\begin{array}{cccc}
1 & 0 & 0 & 0 \\
0 &-1&-1 &0\end{array}
\right)\]  which is isomorphic to  $A_{14,r}$, if $\alpha _1=0$ then $\alpha'_1= 0$ and one can make $\beta'_1=1$ to come to \[A_{15,r}=\left(
\begin{array}{cccc}
0 & 0 & 0 & 0 \\
1 &0& 0 &0
\end{array}
\right).\]

A routine check in each subset case shows that the corresponding algebras presented above are not isomorphic.\end{proof} 
\subsection{Classification of two-dimensional commutative and commutative Jordan algebras}

In this and next subsection we consider some applications of Theorem 2.3.
Commutativity of an 2-dimensional algebra, in terms of its MSC, means equality of its $2^{nd}$ and $3^{rd}$ columns and therefore due to Theorem 2.3 we have the following complete classification of two-dimensional commutative ( not necessarily associative) real algebras.
\begin{thm} Any non-trivial 2-dimensional real commutative  algebra is isomorphic to only one of the following listed, by their matrices of structure constants, commutative algebras:
\[A_{1,r,c}(\mathbf{c})=\left(
\begin{array}{cccc}
\alpha_1 & 0 & 0 & 1 \\
\beta _1& 1-\alpha_1& 1-\alpha_1&0
\end{array}\right), \ \mbox{where}\ \beta_1\geq 0,\ \mathbf{c}=(\alpha_1, \beta_1)\in \mathbb{R}^2,\]
\[A_{2,r,c}(\mathbf{c})=\left(
\begin{array}{cccc}
\alpha_1 & 0 & 0 & -1 \\
\beta _1& 1-\alpha_1& 1-\alpha_1&0
\end{array}\right), \ \mbox{where}\ \beta_1\geq 0,\ \mathbf{c}=(\alpha_1, \beta_1)\in \mathbb{R}^2,\]
\[A_{3,r,c}(c)=\left(
\begin{array}{cccc}
0 & 1 & 1 & 0 \\
\beta _1& 1 & 1&-1
\end{array}\right),\ \mbox{where}\ c=\beta_1\in \mathbb{R},\]
\[A_{4,r,c}(c)=\left(
\begin{array}{cccc}
\alpha _1 & 0 & 0 & 0 \\
0 & 1-\alpha_1& 1-\alpha _1&0
\end{array}\right),\ \mbox{where}\ c=\alpha_1\in \mathbb{R},\]
\[A_{5,r,c}(c)=\left(
\begin{array}{cccc}
\frac{2}{3}& 0 & 0 & 0 \\
1 & \frac{1}{3}&\frac{1}{3}&0
\end{array}\right),\
\ A_{6,r,c}=\left(
\begin{array}{cccc}
0 & 1 & 1 & 0 \\
1 &0&0 &-1
\end{array}\right),\]
\[A_{7,r,c}=\left(
\begin{array}{cccc}
0 & 1 & 1 & 0 \\
-1 &0&0 &-1
\end{array}\right),\
\ A_{8,r,c}=\left(
\begin{array}{cccc}
0 & 1 & 1 & 0 \\
0 &0&0 &-1
\end{array}\right),\
\ A_{9,r,c}=\left(
\begin{array}{cccc}
0 & 0 & 0 & 0 \\
1 &0&0 &0
\end{array}\right).\]\end{thm}
\begin{defn}An algebra $\mathbb{A}$  with multiplication $\cdot$ given by a bilinear map $(\mathbf{u},\mathbf{v})\mapsto \mathbf{u}\cdot \mathbf{v}$ over a field $\mathbb{F}$
is said to be a commutative Jordan algebra if \[\mathbf{u}\cdot \mathbf{v}=\mathbf{v}\cdot \mathbf{u},\ (\mathbf{u}\cdot \mathbf{v})\mathbf{u}^2=\mathbf{u}
\cdot(\mathbf{v}\cdot \mathbf{u}^2),\ \mbox{whenever}\ \mathbf{u}, \mathbf{v}\in \mathbb{A}.\]\end{defn}
 In terms of structure constants matrix $A$ the second identity can be written in the following equivalent form \[(A(A\otimes A)-A(E\otimes A(E\otimes A)))(u\otimes e_i\otimes u\otimes u)=0, i=1,2\] whenever $u=(u_1,u_2)\in \mathbb{R}^2$,  where $e_1=(1,0), e_2=(0,1)$ are column vectors. In $A=\left(\begin{array}{cccc} \alpha_1 & \alpha_2 & \alpha_2 &\alpha_4\\ \beta_1 & \beta_2 & \beta_2 &\beta_4\end{array}\right)$ case it can be checked that it happens if and only if either \begin{equation}\begin{array}{c}\beta_1=\alpha_4=2\beta_2-\alpha_1=2\alpha_2-\beta_4=0,\ \ \mbox{ or}\\ 
\beta_1\alpha_4-\beta_2\alpha_2=\beta^2_2-\beta_1\beta_4+\alpha_2\beta_1-\alpha_1\beta_2=\alpha^2_2-\alpha_1\alpha_4+\alpha_4\beta_2-\alpha_2\beta_4=
0 \end{array}.\end{equation} 
To get a complete classification of two-dimensional commutative Jordan algebras it is enough to list all algebras from the above list of commutative algebras which satisfy (2.2). Direct checking provides the following list of commutative Jordan algebras.
\begin{thm} Any non-trivial 2-dimensional real commutative Jordan algebra is isomorphic to only one of the following listed, by their matrices of structure constants, algebras:
\[A_{2,r}(1/2,0,1/2)=\left(
\begin{array}{cccc}
1/2 & 0 & 0 & 1 \\
0 &1/2&1/2 &0
\end{array}\right),\ \ A_{3,r}(1/2,0,1/2)=\left(
\begin{array}{cccc}
1/2 & 0 & 0 & -1 \\
0 &1/2&1/2 &0
\end{array}\right),\]
\[ A_{5,r}(2/3,1/3)=\left(
\begin{array}{cccc}
2/3 & 0 & 0 & 0 \\
0 &1/3&1/3 &0
\end{array}\right),\ A_{5,r}(1/2,1/2)=\left(
\begin{array}{cccc}
1/2 & 0 & 0 & 0 \\
0 &1/2&1/2 &0
\end{array}\right),\] \[ A_{5,r}(1,0)=\left(
\begin{array}{cccc}
1 & 0 & 0 & 0 \\
0 &0&0 &0
\end{array}\right), \ A_{15,r}=\left(
\begin{array}{cccc}
0 & 0 & 0 & 0 \\
1 &0&0 &0
\end{array}\right).\]\end{thm}
\begin{rem} In \cite{G} a classification of two-dimensional commutative Jordan algebras over algebraically closed fields is given by the following list of  algebras \[ J_{1}=\left(
\begin{array}{cccc}
1 & 0 & 0 & 0 \\
0 &1&1&1
\end{array}\right), J_{2}=\left(
\begin{array}{cccc}
1 & 0 & 0 & 0 \\
0 &1&0&1
\end{array}\right),\
\ J_{3}=\left(
\begin{array}{cccc}
1 & 1/2 & 1/2 & 0 \\
0 &1/2&1/2 &1
\end{array}\right),\] \[J_{4}=\left(
\begin{array}{cccc}
1 & 0 & 0 & 0 \\
0 &1&1 &0
\end{array}\right),\
\ J_{5}=\left(
\begin{array}{cccc}
1 & 0 & 0 & 0 \\
0 &0&0 &0
\end{array}\right), J_{6}=\left(
\begin{array}{cccc}
0 & 0 & 0 & 0 \\
1 &0&0 &0
\end{array}\right).\] It should be noted that in reality in this list $J_1$ and $J_2$ are isomorphic algebras, indeed $J_1=gJ_2(g^{-1})^{\otimes 2}$ at $g=\left(\begin{array}{cc}1&0\\-1&1\end{array}\right)$. Therefore to make that classification valid one of them should be dropped and, moreover,  
assumed that the characteristic of the basic field is not 2,3. 
In case of the field $\mathbb{R}$ we have the corresponding algebras $J_1,J_3,J_4,J_5,J_6$ as $A_{2,r}(1/2,0,1/2),\ A_{5,r}(2/3,1/3),\ A_{5,r}(1/2,1/2),\ A_{5,r}(1,0), \ A_{15,r}$
, respectively.\end{rem}
\subsection{Classification of 2-dimensional real division algebras}

Now we show how one can derive a classification of $2$-dimensional real division algebras from Theorem 2.3. For a basis free approach to this problem one can see \cite{P1} and for coordinate based \cite{A3, B2}.  

\begin{defn} A finite dimensional algebra $\mathbb{A}$ is said to be division algebra if $\mathbf{u}\cdot \mathbf{v}=0$ is valid if and only if when at least one of $\mathbf{u}$, $\mathbf{v}=0$ is zero.\end{defn}
It is clear that for some nonzero $(x,y)$ and nonzero $(w,z)$ the equality 
\[\left(\begin{array}{cccc} \alpha_1 & \alpha_2 & \alpha_3 &\alpha_4\\ \beta_1 & \beta_2 & \beta_3 &\beta_4\end{array}\right)((x,y)\otimes(z,w))=\left(\begin{array}{c} x(z\alpha_1+w\alpha_2) + y(z\alpha_3+w\alpha_4)\\ x(z\beta_1+w\beta_2) + y(z\beta_3+w\beta_4)\end{array}\right)=0\] holds true  if and only if 
\[\left|\begin{array}{cc} z\alpha_1+w\alpha_2& z\alpha_3+w\alpha_4\\ z\beta_1+w\beta_2&z\beta_3+w\beta_4\end{array}\right|= z^2\left|\begin{array}{cc} \alpha_1& \alpha_3\\ \beta_1&\beta_3\end{array}\right|+zw(\left|\begin{array}{cc} \alpha_1& \alpha_4\\ \beta_1&\beta_4\end{array}\right|+\left|\begin{array}{cc} \alpha_2& \alpha_3\\ \beta_2&\beta_3\end{array}\right|)+w^2\left|\begin{array}{cc} \alpha_2& \alpha_4\\ \beta_2&\beta_4\end{array}\right|=0\ .\]  
The last equality holds true for some nonzero $(w,z)$ if and only if $D=\Delta^2_l-4\Delta_l\Delta_r \geq  0,$ where  \[\Delta_l=\left|\begin{array}{cc} \alpha_1& \alpha_3\\ \beta_1&\beta_3\end{array}\right|,\ \Delta_m=\left|\begin{array}{cc} \alpha_1& \alpha_4\\ \beta_1&\beta_4\end{array}\right|+\left|\begin{array}{cc} \alpha_2& \alpha_3\\ \beta_2&\beta_3\end{array}\right|,\  \Delta_r=\left|\begin{array}{cc} \alpha_2& \alpha_4\\ \beta_2&\beta_4\end{array}\right|.\]

It means that algebra given by MSC as $A=\left(\begin{array}{cccc} \alpha_1 & \alpha_2 & \alpha_3 &\alpha_4\\ \beta_1 & \beta_2 & \beta_3 &\beta_4\end{array}\right)$ is division algebra if and only if \begin{equation}D=\Delta^2_m-4\Delta_l\Delta_r< 0.\end{equation}

Now to classify 2-dimensional real division algebras it is enough to list algebras from Theorem 2.3 for which condition (2.3) holds true.
\begin{thm} Any nontrivial two-dimensional real division algebra is isomorphic to only one algebra from the following listed, by their matrices of structure constants, division algebras:
	$A_{1,r}(\mathbf{c})$, for which $\Delta^2_m-4\Delta_l\Delta_r<0$,
	$A_{2,r}(\mathbf{c})$, for which $\beta^2_1+4\alpha_1(1-\alpha_1)\beta_2< 0$, $A_{3,r}(\mathbf{c})$, for which $\beta^2_1-4\alpha_1(1-\alpha_1)\beta_2< 0$,  $A_{4,r}(\mathbf{c})$, for which $(1-\beta_2)^2-4\beta_1< 0,$ 
	$A_{7,r}(\mathbf{c})$, for which $\beta^2_1-4\alpha^2_1(1-\alpha_1)< 0$, 
	$A_{8,r}(\mathbf{c})$, for which $\beta^2_1+4\alpha^2_1(1-\alpha_1)< 0$, $A_{9,r}(\mathbf{c})$, for which $1-4\beta_1< 0$
	and $A_{12,r}$. \end{thm}
\begin{proof} In $A=A_{1,r}(\mathbf{c})=\left(
	\begin{array}{cccc}
	\alpha_1 & \alpha_2 &\alpha_2+1 & \alpha_4 \\ \beta_1 & -\alpha_1 & -\alpha_1+1 & -\alpha_2
	\end{array}\right)$ case one has  \[\Delta_l=\alpha_1(-\alpha_1+1)-\beta_1(\alpha_2+1),\ \Delta_m=-\alpha_1\alpha_2-\alpha_4\beta_1+\alpha_1+\alpha_2,\ \Delta_r=-\alpha^2_2+\alpha_1\alpha_4.\]
	Therefore $A_{1,r}(\mathbf{c})$ is division algebra whenever $D=\Delta^2_l-4\Delta_l\Delta_r < 0.$ 
	
		In $A_{2,r}(\mathbf{c})=\left(
	\begin{array}{cccc}
	\alpha_1 & 0 & 0 & 1 \\
	\beta _1& \beta _2& 1-\alpha_1&0
	\end{array}\right)$ case $\Delta_l=\alpha_1(1-\alpha_1),\ \Delta_m=-\beta_1,\ \Delta_r=-\beta_2$ and therefore only the following $A_{2,r}(\mathbf{c})$, for which $\beta^2_1+4\alpha_1(1-\alpha_1)\beta_2< 0$, are division algebras.
	
	In $A_{3,r}(\mathbf{c})=\left(
	\begin{array}{cccc}
	\alpha_1 & 0 & 0 & -1 \\
	\beta _1& \beta _2& 1-\alpha_1&0
	\end{array}\right)$ case $\Delta_l=\alpha_1(1-\alpha_1),\ \Delta_m=\beta_1,\ \Delta_r=\beta_2$ and therefore only $A_{3,r}(\mathbf{c})$, for which $\beta^2_1-4\alpha_1(1-\alpha_1)\beta_2< 0$, are division algebras.
	
		In $A_{4,r}(\mathbf{c})=\left(
	\begin{array}{cccc}
	0 & 1 & 1 & 0 \\
	\beta _1& \beta _2 & 1&-1
	\end{array}\right)$ case $\Delta_l=-\beta_1,\ \Delta_m=1-\beta_2,\ \Delta_r=-1$ and therefore only $A_{4,r}(\mathbf{c})$, for which $(1-\beta_2)^2-4\beta_1< 0,$  are division algebras.
	
	In $A_{5,r}(\mathbf{c})=\left(
	\begin{array}{cccc}
	\alpha _1 & 0 & 0 & 0 \\
	0 & \beta _2& 1-\alpha _1&0
	\end{array}\right)$ case $\Delta_l=\alpha_1(1-\alpha_1),\ \Delta_m=0,\ \Delta_r=0$ and $D=0$. Therefore among $A_{5,r}(\mathbf{c})$ there is no division algebra.
	
	In $A_{6,r}(\mathbf{c})=\left(
	\begin{array}{cccc}
	\alpha_1& 0 & 0 & 0 \\
	1 & 2\alpha_1-1 & 1-\alpha_1&0
	\end{array}\right)$ case $\Delta_l=\alpha_1(1-\alpha_1),\ \Delta_m=0,\ \Delta_r=0$ and $D=0$. Therefore among $A_{6,r}(\mathbf{c})$ there is no division algebra.
	
	In $A_{7,r}(\mathbf{c})=\left(
	\begin{array}{cccc}
	\alpha_1 & 0 & 0 & 1 \\
	\beta _1& 1-\alpha_1 & -\alpha_1&0
	\end{array}\right)$ case $\Delta_l=-\alpha^2_1,\ \Delta_m=-\beta_1,\ \Delta_r=-(1-\alpha_1)$ and therefore only $A_{7,r}(\mathbf{c})$, for which $\beta^2_1-4\alpha^2_1(1-\alpha_1)< 0$, are division algebras.
	
	In $A_{8,r}(\mathbf{c})=\left(
	\begin{array}{cccc}
	\alpha_1 & 0 & 0 & -1 \\
	\beta _1& 1-\alpha_1 & -\alpha_1&0
	\end{array}\right)$ case $\Delta_l=-\alpha^2_1,\ \Delta_m=\beta_1,\ \Delta_r=1-\alpha_1$ and therefore only $A_{8,r}(\mathbf{c})$, for which $\beta^2_1+4\alpha^2_1(1-\alpha_1)< 0$, are division algebras.
	
	In $A_{9,r}(\mathbf{c})=\left(
	\begin{array}{cccc}
	0 & 1 & 1 & 0 \\
	\beta_1 & 1 & 0&-1
	\end{array}\right)$ case $\Delta_l=-\beta_1,\ \Delta_m=-1,\ \Delta_r=-1$, so only $A_{9,r}(\mathbf{c})$, for which $1-4\beta_1< 0$, are division algebras.
	
	In $A_{10,r}(\mathbf{c})=\left(
	\begin{array}{cccc}
	\alpha_1 & 0 & 0 & 0 \\
	0& 1-\alpha_1 & -\alpha_1&0
	\end{array}\right)$ case $\Delta_l=-\alpha^2_1,\ \Delta_m=0,\ \Delta_r=0$ and $D=0$. Therefore among $A_{10,r}(\mathbf{c})$ there is no division algebra.
	
	In $A_{11,r}=\left(
	\begin{array}{cccc}
	\frac{1}{3}& 0 & 0 & 0 \\
	1 & \frac{2}{3} & -\frac{1}{3}&0
	\end{array}\right)$ case $\Delta_l=-1/9,\ \Delta_m=0,\ \Delta_r=0$ and $D=0$. Therefore $A_{11,r}$ is not division algebra.
	
	In $ A_{12,r}=\left(
	\begin{array}{cccc}
	0 & 1 & 1 & 0 \\
	1 &0& 0 &-1
	\end{array}
	\right)$ case $\Delta_l=-1,\ \Delta_m=0,\ \Delta_r=-1$ and $D=-4$, so $A_{12,r}$ is a division algebra.
	
	In $ A_{13,r}=\left(
	\begin{array}{cccc}
	0 & 1 & 1 & 0 \\
	-1 &0& 0 &-1
	\end{array}
	\right)$ case $\Delta_l=1,\ \Delta_m=0,\ \Delta_r=-1$ and $D=4$, so $A_{13,r}$ algebra has divisors of zero.
	
		In $ A_{14,r}=\left(
	\begin{array}{cccc}
	0 & 1 & 1 & 0 \\
	0 &0& 0 &-1
	\end{array}
	\right)$ case $\Delta_l=0,\ \Delta_m=0,\ \Delta_r=-1$ and $D=0$, so $A_{14,r}$ is not division algebra.
	
	In $A_{15,r}=\left(
	\begin{array}{cccc}
	0 & 0 & 0 & 0 \\
	1 &0&0 &0\end{array}
	\right)$ case $\Delta_l=0,\ \Delta_m=0,\ \Delta_r=0$, so $A_{15,r}$ is not  division algebra.\end{proof}
\section{Classification of two-dimensional real evolution algebras}
\begin{defn} An $n$-dimensional algebra $\mathbb{E}$ is said to be an evolution algebra if it admits a basis $\{e^1,e^2,...,e^n\}$ such that $e^ie^j=0$ whenever $i\neq j,\ i,j=1,2,...,n$.\end{defn}

We are going to prove the following result on 2-dimensional real evolution algebras.
\begin{thm} Any nontrivial 2-dimensional real evolution algebra is isomorphic to only one algebra from the following listed, by their matrices of structure constants, evolution algebras:
	\[E_{1,r}(c,b)\simeq E_{1,r}(b,c)=\left(\begin{array}{cccc} 1 & 0 &0 &b\\ c & 0 & 0 &1\end{array}\right),\ \mbox{ where}\ bc\neq 1,\ (b, c)\in\mathbb{R}^2,\]
	\[E_{2,r}(b)=\left(\begin{array}{cccc} 1 & 0 &0 &b\\ 1 & 0 & 0 &0\end{array}\right),\ \mbox{ where}\ b\in\mathbb{R},\
	 E_{3,r}=\left(\begin{array}{cccc} 0 & 0 &0 &1\\ 1 & 0 & 0 &0\end{array}\right),\ 
  \ E_{4,r}=\left(\begin{array}{cccc} 1 & 0 &0 &1\\ 0 & 0 & 0 &0\end{array}\right),\] \[ E_{5,r}=\left(\begin{array}{cccc} 1 & 0 &0 &-1\\ 0 & 0 & 0 &0\end{array}\right),\
  E_{6,r}=\left(\begin{array}{cccc} 1 & 0 &0 &-1\\ -1 & 0 & 0 &1\end{array}\right),\
	\ E_{7,r}=\left(\begin{array}{cccc} 0 & 0 &0 &1\\ 0 & 0 & 0 &0\end{array}\right).\]
\end{thm}
\begin{proof} Let $\mathbb{E}$ be a nontrivial real evolution algebra given by
$E=\left(\begin{array}{cccc} a & 0 & 0 &b\\ c & 0 & 0 &d\end{array}\right)$ and \\ $E'=\left(\begin{array}{cccc} \alpha'_1 & \alpha'_2 & \alpha'_3 &\alpha'_4\\ \beta'_1 & \beta'_2 & \beta'_3 &\beta'_4\end{array}\right)=gE(g^{-1})^{\otimes 2}$, where  $g^{-1}=\left(\begin{array}{cccc} \xi_1& \eta_1\\ \xi_2& \eta_2\end{array}\right)$. For entries of $E'$ one has\\
$\alpha'_1=\frac{1}{\Delta}(\xi^2_1(a\eta_2-c\eta_1)+\xi^2_2(b\eta_2-d\eta_1)),$\\
$\alpha'_2 = \alpha'_3= \frac{1}{\Delta}(\xi_1\eta_1(a\eta_2-c\eta_1)+\xi_2\eta_2(b\eta_2-d\eta_1)),$\\
$\alpha'_4=\frac{1}{\Delta}(\eta^2_1(a\eta_2-c\eta_1)+\eta^2_2(b\eta_2-d\eta_1)),$\\
$\beta'_1=\frac{1}{\Delta}(\xi^2_1(-a\xi_2+c\xi_1)+\xi^2_2(-b\xi_2+d\xi_1)),$\\
$\beta'_2 = \beta'_3=\frac{1}{\Delta}(\xi_1\eta_1(-a\xi_2+c\xi_1)+\xi_2\eta_2(-b\xi_2+d\xi_1)),$\\$\beta'_4=\frac{1}{\Delta}(\eta^2_1(-a\xi_2+c\xi_1)+\eta^2_2(-b\xi_2+d\xi_1)),$ where $\Delta=\xi_1\eta_2-\xi_2\eta_1$.

In particular one has
\[\left(\begin{array}{c} \alpha'_2\\ \beta'_2\end{array}\right)=
\left(\begin{array}{cc} \xi_1 & \eta_1\\ \xi_2 & \eta_2\end{array}\right)^{-1}\left(\begin{array}{cc} a & b\\ c & d\end{array}\right)\left(\begin{array}{c} \xi_1\eta_1\\ \xi_2\eta_2\end{array}\right).\]  Note also that \[\left(\begin{array}{cc} \alpha'_1 & \alpha'_4\\ \beta'_1 & \beta'_4\end{array}\right)=\left(\begin{array}{cc} \xi_1 & \eta_1\\ \xi_2 & \eta_2\end{array}\right)^{-1}\left(\begin{array}{cc} a & b\\ c & d\end{array}\right)\left(\begin{array}{cc}  \xi^2_1 & \eta^2_1\\ \xi^2_2 & \eta^2_2\end{array}\right),\] which shows that $\alpha'_1\beta'_4-\alpha'_4\beta'_1=0$ whenever $ad-bc=0$.

We are going to make \begin{equation}
\alpha'_2 = \alpha'_3=\beta'_2 = \beta'_3=0
\end{equation} and $\alpha'_1,\ \alpha'_4,\ \beta'_1,\ \beta'_4$ as simple as possible. For that we have to consider several cases. 

1. $ad-bc\neq 0$. In this case (3.1) is equivalent to $\xi_1\eta_1= \xi_2\eta_2=0$. Let us consider $g=\left(\begin{array}{cc} \xi_1 & 0\\ 0 & \eta_2\end{array}\right)$. In this case $\Delta=\xi_1\eta_2$ and $ \alpha'_1=a\xi_1,\
\alpha'_4=b\frac{\eta^2_2}{\xi_1},\
\beta'_1=c\frac{\xi^2_1}{\eta_2},\
\beta'_4=d\eta_2.$

Due to $ad-bc\neq 0$ one has the following possibilities.

1-a. $a\neq 0, d\neq 0$. In this case one can make $ \alpha'_1=1,  \beta'_4=1$ to get, after obvious re notation, $E_{1,r}(b,c)=\left(\begin{array}{cccc} 1 & 0 &0 &b\\ c & 0 & 0 &1\end{array}\right)$, where $bc\neq 1$.

1-b. $a\neq 0, d= 0$. In this case $\beta'_4=0 $ and one can make $ \alpha'_1=1,  \beta'_1=1$ to get, after obvious re notation, $E_{2,r}(b)=\left(\begin{array}{cccc} 1 & 0 &0 &b\\ 1 & 0 & 0 &0\end{array}\right)$, where $b\neq 0$.

1-c. $a=0, d\neq 0$. In this case $ \alpha'_1=0$ and one can make $\beta'_4=1,  \alpha'_4=1$ to get, after obvious renotation, $E'=\left(\begin{array}{cccc} 0 & 0 &0 &1\\ c & 0 & 0 &1\end{array}\right)$, where $c\neq 0$. It is isomorphic to $E_{2,r}(c)$.

1-d. $a=0, d= 0$. In this $ \alpha'_1=0,  \beta'_4=0$ and one can make $\beta'_1=\alpha'_4=1$ to get \[E_{3,r}=\left(\begin{array}{cccc} 0 & 0 &0 &1\\ 1 & 0 & 0 &0\end{array}\right).\]

2. $ad-bc= 0$.

2-a. Both $(a,b)$, $(c,d)$ are nonzero and $(c,d)=\lambda (a,b)$ case.  In this case
(3.1) is equivalent to \[a\xi_1\eta_1+b\xi_2\eta_2=0,\
 \alpha'_1=\frac{\eta_2-\lambda\eta_1}{\Delta}(a\xi^2_1+b\xi^2_2),\
\alpha'_4=\frac{\eta_2-\lambda\eta_1}{\Delta}(a\eta^2_1+b\eta^2_2),\]
\[\beta'_1=-\frac{\xi_2-\lambda\xi_1}{\Delta}(a\xi^2_1+b\xi^2_2),\
\beta'_4=-\frac{\xi_2-\lambda\xi_1}{\Delta}(a\eta^2_1+b\eta^2_2).\]

If $a+b\lambda^2\neq 0$ make $\xi_2-\lambda\xi_1=0$. Then $\xi_1\neq 0$, equality $a\xi_1\eta_1+b\xi_2\eta_2=\xi_1(a\eta_1+b\lambda\eta_2)$ implies $a\eta_1+b\lambda\eta_2=0$, in particular $\frac{\eta_2}{\eta_1}=-\frac{a}{b\lambda}$,  $\Delta=\xi_1(\eta_2-\lambda\eta_1)$ and \[\beta'_1=\beta'_4=0,  \alpha'_1=(a+b\lambda^2)\xi_1, \alpha'_4=\frac{\eta^2_1}{\xi_1}\frac{a(a+b\lambda^2)}{b\lambda^2}.\]

It implies that if $b\neq 0$ then one can make $\alpha'_1=1$, and $\alpha'_4=1$ or $-1$ or $0$, depending on $\frac{a}{b\lambda^2}$, to get  
\[E_{4,r}=\left(\begin{array}{cccc} 1 & 0 &0 &1\\ 0 & 0 & 0 &0\end{array}\right)\ \ \mbox{or}\ \
E_{5,r}=\left(\begin{array}{cccc} 1 & 0 &0 &-1\\ 0 & 0 & 0 &0\end{array}\right)\ \ \mbox{or}\ E'=\left(\begin{array}{cccc} 1 & 0 &0 &0\\ 0 & 0 & 0 &0\end{array}\right).\] The last $E'$ is isomorphic to $E_{2,r}(0)$. If $b=0$ then $\eta_1$ have to be zero and $\alpha'_1=a\xi_1$,\ $\alpha'_4=a\frac{\eta^2_1}{\xi_1}=0$. So in this case one can make $\alpha'_1=1$ to come to the same $E'$.

If $a+b\lambda^2= 0$, note in this case $a,b, \lambda$ have to be nonzero, make $\xi_2=\eta_1=0$. Then $\Delta=\xi_1\eta_2$,
and \[   \alpha'_1=a\xi_1,
\alpha'_4=\frac{b\eta^2_2}{\xi_1},
\beta'_1=\frac{a\lambda\xi^2_1}{\eta_2}, \beta'_4=b\lambda\eta_2.\]    
It implies that one can make $\alpha'_1=1, \beta'_4=1$ to get $\alpha'_4=\frac{a}{b\lambda^2}=-1, \beta'_1=\frac{b\lambda^2}{a}=-1$ and 
\[E_{6,r}=\left(\begin{array}{cccc} 1 & 0 &0 &-1\\ -1 & 0 & 0 &1\end{array}\right).\]

2-b. $c=d= 0$. In this case
\[ \alpha'_1=\frac{\eta_2}{\Delta}(a\xi^2_1+b\xi^2_2),\ \
\alpha'_2 = \alpha'_3= \frac{\eta_2}{\Delta}(a\xi_1\eta_1+b\xi_2\eta_2),\ \
\alpha'_4=\frac{\eta_2}{\Delta}(a\eta^2_1+b\eta^2_2),\]
\[\beta'_1=-\frac{\xi_2}{\Delta}(a\xi^2_1+b\xi^2_2),\ \
\beta'_2 = \beta'_3=-\frac{\xi_2}{\Delta}(a\xi_1\eta_1+b\xi_2\eta_2),\ \ \beta'_4=-\frac{\xi_2}{\Delta}(a\eta^2_1+b\eta^2_2).\]

Taking $\xi_2=0$, $\eta_1=0$ results in
\[ \alpha'_1=a\xi_1,\ \alpha'_2=\alpha'_3=0,\ \alpha'_4=\frac{b\eta^2_2}{\xi_1},\
\beta'_1=\beta'_2=\beta'_3=\beta'_4=0.\] 

If $a\neq 0$ then one can make $\alpha'_1=1$ and $\alpha'_4 =1$ or $-1$ or $0$, depending on $\frac{b}{a}$ to get 
\[E_{4,r}=\left(\begin{array}{cccc} 1 & 0 &0 &1\\ 0 & 0 & 0 &0\end{array}\right)\ \mbox{or}\ 
E_{5,r}=\left(\begin{array}{cccc} 1 & 0 &0 &-1\\ 0 & 0 & 0 &0\end{array}\right)\ \mbox{or}\ 
E'=\left(\begin{array}{cccc} 1 & 0 &0 &0\\ 0 & 0 & 0 &0\end{array}\right),\]  respectively. The last $E'$ is isomorphic to $E_{2,r}(0)$.

If $a= 0$ then 
\[ \alpha'_1=0,\ \alpha'_2=\alpha'_3=0,\ \alpha'_4=\frac{b\eta^2_2}{\xi_1},\
\beta'_1=\beta'_2=\beta'_3=\beta'_4=0,\] and one can make $\alpha'_4=1$ to get
$E_{7,r}=\left(\begin{array}{cccc} 0 & 0 &0 &1\\ 0 & 0 & 0 &0\end{array}\right)$. 

2-c. $a=b= 0$. In this case
\[ \alpha'_1=-\frac{\eta_1}{\Delta}(c\xi^2_1+d\xi^2_2),\ \
\alpha'_2 = \alpha'_3= -\frac{\eta_1}{\Delta}(c\xi_1\eta_1+d\xi_2\eta_2),\ \
\alpha'_4=-\frac{\eta_1}{\Delta}(c\eta^2_1+d\eta^2_2),\]
\[\beta'_1=\frac{\xi_1}{\Delta}(c\xi^2_1+d\xi^2_2),\ \
\beta'_2 = \beta'_3=\frac{\xi_1}{\Delta}(c\xi_1\eta_1+d\xi_2\eta_2),\ \ \beta'_4=\frac{\xi_1}{\Delta}(c\eta^2_1+d\eta^2_2),\] which is similar to $c=d=0$ case equalities. A justification, similar to $c=d=0$ case, shows that such algebras are isomorphic to one of previously considered algebras.\end{proof}
\begin{rem} Our classification of $2$-dimensional real evolution algebras agrees with a similar classification given in \cite{M}, they approach also is coordinate based.\end{rem}  
\subsection{The groups of automorphisms and derivation algebras}

Let  $I$ stand for the second identity matrix $\left(
\begin{array}{cc}
1 & 0 \\
0 & 1 
\end{array}
\right).$
If $\mathbb{E}$ is an algebra given by MSC $E$ then its group of automorphisms  $Aut(E)$ can be presented as 
\[Aut(E)=\{g=\left(
\begin{array}{cc}
x & y \\
z & t \\
\end{array}
\right): xt-yz\neq 0 \ \mbox{ and } \  gE-E(g\otimes g)=0.\}\]

\begin{thm} The following equalities are true. \[Aut(E_{1,r}(b,c))=\{I\},\ \mbox{if}\  \mbox{and}\ b\neq c, 
	\ Aut(E_{1,r}(b,b))=\{I, \left(\begin{array}{cc}0&1\\ 1&0\end{array}\right)\},\ \mbox{if}\ b^2\neq 1,\]
\[ Aut(E_{2,r}(b))=\{I\},\ \mbox{if}\ b\neq 0, \
	\ Aut(E_{2,r}(0))=\{\left(
	\begin{array}{cc}
	1 & 0 \\
	t & 1-t \\
	\end{array}
	\right):\ t\neq 1 \}, \]		\[ Aut(E_{3,r})=\{I, \left(
	\begin{array}{cc}
	0 & 1 \\
	1 & 0 \\
	\end{array}
	\right)\}, \ 
	\ Aut(E_{4,r})=Aut(E_{5,r})=\{I, \left(
	\begin{array}{cc}
	1 & 0 \\
	0 & -1 \\
	\end{array}
	\right) \}, \]
	\[Aut(E_{6,r})=\{ \left(\begin{array}{cc}t&1-t\\ 1-t&t\end{array}\right):\ \ t\neq \frac{1}{2}\},\
		\ Aut(E_{7,r})=\{\left(
	\begin{array}{cc}
	t^2 & s \\
	0 & t \\
	\end{array}
	\right) :\ \ t\neq 0,\ s\in \mathbb{R} \}. \]\end{thm}
Proof of this result consists of finding out all nonsingular solutions of the equation $gE-E(g\otimes g)=0$ for each $E$ taken from $E_{1,r}(b,c),\ E_{2,r}(b),\ E_{3,r},\ E_{4,r},\ E_{5,r},\ E_{6,r}$  and $E_{7,r}$.

If $\mathbb{E}$ is an algebra given by MSC $E$ then algebra of its derivations $Der(\mathbb{E})$ can be presented as following
\[ Der(\mathbb{E})=\{D=\left(\begin{array}{cc}
x & y \\
z & t \\
\end{array}\right):\  E(D\otimes I+I\otimes D)-DE=0\}\] 

\begin{thm}.The following equalities are true 
	\[Der(E_{1,r}(b,c))=\{0\},\ 
	\ Der(E_{2,r}(b))=\{0\},\ \mbox{if}\  b\neq 0,\]
	\[ Der(E_{2,r}(0))=\{\left(
	\begin{array}{cc}
	0 & 0 \\
	t & -t \\
	\end{array} \right):\ t\in \mathbb{R}\} ,\ \ Der(E_{3,r})=Der(E_{4,r})=Der(E_{5,r})=\{0\},\]
	\[ Der(E_{6,r}=\{\left(
	\begin{array}{cc}
	-t & t \\
	t & -t \\
	\end{array} \right):\ t\in \mathbb{R}\} ,\ \ Der(E_{7,r})=\{\left(
	\begin{array}{cc}
	2t & s \\
	0 & t \\
	\end{array} \right):\ t,s\in \mathbb{R}\}.\]\end{thm}
Proof of this result consists of finding out all solutions $D$ of the equation $E(D\otimes I+I\otimes D)-DE=0$ for each $E$ taken from $E_{1,r}(b,c),\ E_{2,r}(b),\ E_{3,r},\ E_{4,r},\ E_{5,r},\ E_{6,r}$ and $E_{7,r}$.

\begin{center}{\textbf{Acknowledgments}}
\end{center} This research is supported by FRGS14-153-0394, MOHE.


\begin{thebibliography}{99}
\bibitem{A1} H. Ahmed, U. Bekbaev, I. Rakhimov, Comlete classification of 2-dimensional algebras, \textit{arXiv 1702.08616 } (2017).
\bibitem{A2} S.C. Althoen and K.D.Hansen. Two-dimensional real algebras with zero divisors. \textit{Acta Sci.Math.(Szeged)} \textbf{56}(1992),23--42.
\bibitem{A3} S.C. Althoen  and Kugler,L.D. When is $\mathbf{R^2}$ a division algebra? \textit{Amer.Math.Monthly} \textbf{90}(1983),625--635.
\bibitem{B} U. Bekbaev, On classification of finite dimensional algebras, \emph{ arXiv:1504.01194} (2015).
\bibitem{B1} J.M. Berm$\acute{u}$dez, J. Fres$\acute{a}$n,J. S.  Hern$\acute{a}$ndez, On the variety of two dimensional real associative algebras. \textit{Int.J.Contemp.Math.Sciences}, \textbf{2}(26) (2007), 1293--1305.
\bibitem{B2} I. Burdujan, Types of nonisomorphic two-dimensional real division algebras. Proceedings of the national conference on algebra (Romanian)(la/c si,1984) An./c t. Univ. "al. I.Cuza" Ia/c si Sec/c t. I a Mat.(N.S) \textbf{31}(1985),92--105.
\bibitem{C}  A.L. Cali and M. Josephy, Two-dimensional real division algebras. \textit{Rev.Un.Mat. Argentina} \textbf{32}(1985),53--63.
\bibitem{D} R. Dur\'{a}n D\'{\i}az, J. Mu\~{n}oz Mesqu\'{e}, A. Peinado Dom\'{\i}nques, Classifying quadratic maps from plane to plane, \textit{Linear Algebra and its Applications}, \textbf{364} (2003), 1--12.
\bibitem{H} H. Encinas, A. Martín del Rey, J. Mu\~{n}oz Mesqu\'{e}, Non-degenerate bilinear alternating maps $f:V\times V\rightarrow V,$  $dim(V)=3$ over an algebraically closed field, \textit{Linear Algebra and its Applications}, \textbf{387} (2004), 69--82.
\bibitem{G} M. Goze, E. Remm, 2-dimensional algebras, \textit{African Journal of Mathematical Physics}, \textbf{10} (2011), 81--91.
\bibitem{P1} H.P. Petersson and M. H$\ddot{u}$bner, Two-dimensional real division algebras revisited.\textit{ Beiträge zur Algebra und Geometrie} \textbf{45}(1) (2004), 29--36. 
\bibitem{P2} H.P. Petersson, The classification of two-dimensional nonassicative algebras, \textit{Result. Math.} \textbf{3} (2000), 120--154.
\bibitem{P3} V. Popov, Generic Algebras: Rational Parametrization and Normal Forms, \textit{arXiv: 1411.6570v2} (2014).
\bibitem{M} Sh.N. Murodov, Classification of two-dimensional real evolution algebras and dynamics of some two-dimensional chains of evolution algebras, \textit{arXiv: 1305.6416v2 [math.DS]} (2013).


\end{thebibliography}
\end{document}